\documentclass[12pt, reqno]{amsart}
\usepackage{amsmath, amsthm, amscd, amsfonts, amssymb, graphicx, color, mathrsfs}
\usepackage[bookmarksnumbered, colorlinks, plainpages]{hyperref}
\usepackage[all]{xy}
\usepackage{slashed}

\usepackage{soul}
\usepackage{cancel}
\usepackage{ulem}

\newtheorem{theorem}{Theorem}[section]

\theoremstyle{definition}

\theoremstyle{remark}

\numberwithin{equation}{section}

\begin{document}
\setcounter{page}{1}

\title[  A brief description of  pseudo-multipliers on Schatten-von Neumann classes ]{ A brief description of  operators associated to the quantum harmonic oscillator on Schatten-von Neumann classes}

\author[D. Cardona]{Duv\'an Cardona}
\address{
  Duv\'an Cardona:
  \endgraf
  Department of Mathematics  
  \endgraf
  Pontificia Universidad Javeriana.
  \endgraf
  Bogot\'a
  \endgraf
  Colombia
  \endgraf
  {\it E-mail address} {\rm duvanc306@gmail.com;
cardonaduvan@javeriana.edu.co}
  }



\begin{abstract}
In this note we  study pseudo-multipliers associated to the harmonic oscillator (also called Hermite multipliers) belonging to Schatten classes on $L^2(\mathbb{R}^n)$.  We also investigate the spectral  trace of these operators.
\end{abstract} \maketitle
\section{Introduction}

\subsection{Outline of the paper} Pseudo-multipliers and multipliers associated to the harmonic oscillator arise from the study of Hermite expansions for complex functions on $\mathbb{R}^n$ (see Thangavelu \cite{Thangavelu,Thangavelu2,Thangavelu3,Thangavelu4,thangavelu5,Thangavelu6}, Epperson \cite{Epperson} and Bagchi and Thangavelu  \cite{BagchiThangavelu}). In this note, we are interested in the membership  of  pseudo-multipliers associated to the harmonic oscillator (also called Hermite pseudo-multipliers) in the Schatten classes, $S_{r}(L^2)$  on $L^2(\mathbb{R}^n)$. With this paper we finish the classification of pseudo-multipliers in classes of $r$-nuclear operators on $L^p$-spaces (see Barraza and Cardona \cite{BarrazaCardona,CardonaBarraza2}), which on $L^2(\mathbb{R}^n)$  coincide with the Schatten-von Neumann classes of order $r$ .  Our main result is Theorem \ref{MainTheorem} where we establish some criteria in order that  pseudo-multipliers belong to the classes $S_{r}(L^2),$ $0<r\leq 2$.  In order to present our main result we recall some notions. Let us consider the sequence of Hermite function on $\mathbb{R}^n,$
\begin{equation}
\phi_\nu=\Pi_{j=1}^n\phi_{\nu_j},\,\,\, \phi_{\nu_j}(x_j)=(2^{\nu_j}\nu_j!\sqrt{\pi})^{-\frac{1}{2}}H_{\nu_j}(x_j)e^{-\frac{1}{2}x_j^2}
\end{equation} 
where $x=(x_1,\cdots,x_n)\in\mathbb{R}^n$, $\nu=(\nu_1,\cdots,\nu_n)\in\mathbb{N}^n_0,$ and $H_{\nu_j}(x_j)$ denotes the Hermite polynomial of order $\nu_j$. It is well known that the Hermite functions provide a complete and orthonormal system in $L^2(\mathbb{R}^n).$ If we consider the operator $L=-\Delta+|x|^{2}$ acting on the Schwartz space $\mathscr{S}(\mathbb{R}^n),$ where $\Delta$ is the standard Laplace operator on $\mathbb{R}^n,$ then we have the relation 
$
L\phi_\nu=\lambda_\nu \phi_\nu,\,\,\nu\in\mathbb{N}_0^n.
$
The operator $L$ is symmetric and positive in $L^2(\mathbb{R}^n)$ and admits a self-adjoint extension $H$ whose domain is given by
\begin{equation}
\textnormal{Dom}(H)=\left\{\sum_{\nu\in\mathbb{N}_0^n}\langle f,\phi_\nu \rangle_{L^2}\phi_\nu:   \sum_{\nu\in\mathbb{N}_0^n}|\lambda_\nu\langle f,\phi_\nu \rangle_{L^2}|^2 <\infty \right\}. 
\end{equation} So, for $f\in \textnormal{Dom}(H), $ we have
\begin{equation}
(Hf)(x)=\sum_{\nu\in\mathbb{N}_{0}}\lambda_\nu\widehat{f}(\phi_\nu)\phi_\nu(x),\,\,\,\, \widehat{f}(\phi_\nu)=\langle f,\phi_\nu \rangle_{L^2}.
\end{equation}
The operator $H$  is precisely the quantum harmonic oscillator on $\mathbb{R}^n$ (see \cite{Prugovecki}). The sequence $\{\widehat{f}(\phi_v)\} $ determines  the Fourier-Hermite transform of $f,$ with corresponding inversion formula
\begin{equation}\label{inversion}
f(x)=\sum_{\nu\in \mathbb{N}^n_0}\widehat{f}(\phi_v)\phi_\nu(x).
\end{equation}
On the other hand, pseudo-multipliers are defined by the quantization process that associates to a function $m$  on $\mathbb{R}^n\times\mathbb{N}_0^n$ a linear operator $T_m$ of the form:
\begin{equation}\label{pseudo}
T_mf(x)=\sum_{\nu\in\mathbb{N}^n_0}m(x,\nu)\widehat{f}(\phi_\nu)\phi_\nu(x),\,\,\,\,f\in \textnormal{Dom}(T_m).
\end{equation}
The  function $m$ on $\mathbb{R}^n\times \mathbb{N}_0^n$ is called the symbol of the pseudo-multiplier $T_m.$ If in \eqref{pseudo}, $m(x,\nu)=m(\nu)$ for all $x,$ the operator $T_m$ is called a multiplier. Multipliers and pseudo-multipliers have been studied, for example, in the works \cite{BagchiThangavelu,stempak,stempak1,stempak2,Thangavelu,Thangavelu2} (and references therein) principally by its mapping properties on $L^p$ spaces. In order that the operator $T_m:L^{p_1}(\mathbb{R}^n)\rightarrow L^{p_2}(\mathbb{R}^n)$ belongs to the Schatten class $S_r(L^2)$, in this paper we  provide some conditions on the symbol $m.$

\subsection{Pseudo-multipliers in Schatten classes}  By following A. Grothendieck \cite{GRO}, we can recall that a linear operator $T:E\rightarrow F$  ($E$ and $F$ Banach spaces) is  $r$-nuclear, if
there exist  sequences $(e_n ')_{n\in\mathbb{N}_0}$ in $ E'$ (the dual space of $E$) and $(y_n)_{n\in\mathbb{N}_0}$ in $F$ such that
\begin{equation}\label{nuc}
Tf=\sum_{n\in\mathbb{N}_0} e_n'(f)y_n,\,\,\, \textnormal{ and }\,\,\,\sum_{n\in\mathbb{N}_0} \Vert e_n' \Vert^r_{E'}\Vert y_n \Vert^r_{F}<\infty.
\end{equation}
\noindent The class of $r-$nuclear operators is usually endowed with the quasi-norm
\begin{equation}
n_r(T):=\inf\left\{ \left\{\sum_n \Vert e_n' \Vert^r_{E'}\Vert y_n \Vert^r_{F}\right\}^{\frac{1}{r}}: T=\sum_n e_n'\otimes y_n \right\}
\end{equation}
\noindent  In addition, when $E=F$ is a Hilbert space and $r=1$ (resp. $r=2$) the definition above agrees with the concept of  trace class operators (resp. Hilbert-Schmidt). For the case of Hilbert spaces $H$, the set of $r$-nuclear operators agrees with the Schatten-von Neumann class of order $r$ (see Pietsch  \cite{P,P2}). We recall that a linear operator $T$ on a Hilbert space $H$ belong to the Schatten class of order $r,$ $S_r(H)$ if
\begin{equation}
s_{r}(T):=\sum_{n\in\mathbb{N}_0}\lambda_n(T)^r<\infty,
\end{equation} where $\{\lambda_n(T)\}$ denotes the sequence of singular values of $T,$ which are the eigenvalues of the operator $\sqrt{T^*T}.$ It was proved in  \cite{BarrazaCardona} that a multiplier $T_m$ with symbol satisfying  conditions of the form
\begin{itemize}
\item 
\begin{equation}
\varkappa(m,p_1,p_2):=\sum_{s=0}^n\sum_{\nu\in I_s}\alpha_{r,p_1,p_2}(s,\nu)|m(\nu)|^r<\infty,
\end{equation}
\end{itemize}
where $\{I_s\}_{s=0}^n$ is a suitable partition of $\mathbb{N}_0^n,$ and $\alpha_{r,p_1,p_2}(s,\nu)$ is a suitable kernel, can be extended to a $r$-nuclear operator from $L^{p_1}(\mathbb{R}^n)$ into $L^{p_2}(\mathbb{R}^n).$  
Although is easy to see that similar necessary conditions apply for pseudo-multipliers, the $r$-nuclearity for these operators in $L^p$-spaces was characterized in \cite{CardonaBarraza2} by the following condition,
\begin{itemize}
\item a pseudo-multiplier $T_{m}$ can be extended to a $r$-nuclear operator from $L^{p_1}$ into $L^{p_2}$ if and only if there exist functions $h_k$ and $g_k$ satisfying \begin{equation}\label{admit}
m(x,\nu)=\phi_{\nu}(x)^{-1}\sum_{k=1}^{\infty}h_{k}(x)\widehat{g}(\phi_\nu),\,\,\,a.e.w.x,\textnormal{ with } 
\sum_{k=0}^{\infty}\Vert g_k\Vert^r_{L^{p_1'}}\Vert h_{k}\Vert^r_{L^{p_2}}<\infty.
\end{equation}
\end{itemize}
If we consider  $p_1=p_2=2,$ and a multiplier $T_m,$ the conditions above can be replaced by the following more simple one,
\begin{equation}
\varkappa(m,2,r):=\sum_{\nu\in\mathbb{N}_0}|m(\nu)|^r<\infty,
\end{equation}
because the set of singular values of a multiplier $T_m$ consists of the elements in the sequence $\{|m(\nu)|\}_{\nu\in\mathbb{N}_0^n}.$ The condition \eqref{admit} characterizes the membership of pseudo-multipliers in Schatten classes in terms of the existence of certain measurable functions. However, in this paper we provide explicit conditions on $m$ in order to guarantee that $T_m\in S_r(L^2),$ because explicit conditions allow us to known information about the distribution of the spectrum of these operators. Our main result is the following theorem.
\begin{theorem}\label{MainTheorem} Let $T_m$ be a pseudo-multiplier with symbol $m$ defined on $\mathbb{R}^n\times \mathbb{N}_0^n.$ Then we have,
\begin{itemize}
\item $T_m$ is a Hilbert-Schmidt operator on $L^2(\mathbb{R}^n),$ i.e., $T_m\in S_2(L^2),$ if and only if
\begin{equation}
\sum_{\nu\in\mathbb{N}_0^n}\int_{\mathbb{R}^n}|m(x,\nu)|^2\phi_\nu(x)^2dx<\infty. 
\end{equation}
\item If $T_m$ is a positive and self-adjoint operator, then $T_m$ is trace class, i.e., $T_m\in S_1(L^2),$  if and only if
\begin{equation}
\sum_{\nu\in\mathbb{N}_0^n}\int_{\mathbb{R}^n}m(x,\nu)\phi_\nu(x)^2dx<\infty.
\end{equation}\label{0r1}
\item $T_m\in S_{r}(L^2),$ $0<r\leq 1,$ if 
\begin{equation}
\sum_{\nu\in\mathbb{N}_0^n}\left(\int_{\mathbb{R}^n}|m(x,\nu)|^2\phi_{\nu}(x)^2dx\right)^{\frac{r}{2}}<\infty.
\end{equation}
\item If $1<r< 2$ and there exists $\sigma>n(\frac{1}{r}-\frac{1}{2})$ such that
\begin{equation}\label{1r2}
\sum_{\nu\in\mathbb{N}_0^n}|\nu|^{2\sigma}\int_{\mathbb{R}^n}|m(x,\nu)|^2\phi_\nu(x)^2dx<\infty,
\end{equation}then $T_m\in S_{r}(L^2).$
\end{itemize}
 \end{theorem}

\subsection{Related works}
Now, we include some references on the subject. Sufficient conditions for the  $r$-nuclearity of spectral multipliers associated to the harmonic oscillator, but, in modulation spaces and Wiener amalgam spaces have been considered by  J. Delgado, M. Ruzhansky and B. Wang in \cite{DRB,DRB2}. The Properties of these multipliers in $L^p$-spaces have been investigated in the references S. Bagchi, S. Thangavelu  \cite{BagchiThangavelu}, J. Epperson \cite{Epperson},  K. Stempak and J.L. Torrea \cite{stempak,stempak1,stempak2},  S. Thangavelu \cite{Thangavelu,Thangavelu2} and references therein. Hermite expansions for distributions can be found in B. Simon \cite{Simon}. The $r$-nuclearity and Grothendieck-Lidskii formulae for multipliers and other types of integral operators can be found in \cite{D3,DRB2}.   On Hilbert spaces the class of $r$-nuclear operators agrees with the Schatten-von Neumann class $S_{r}(H);$ in this context operators with integral kernel on Lebesgue spaces and, in particular, operators with  kernel acting of a special way with anharmonic oscillators of the form $E_a=-\Delta_x+|x|^a,$ $a>0,$ has been considered on  Schatten classes on $L^2(\mathbb{R}^n)$ in J. Delgado and M. Ruzhansky \cite{kernelcondition}.
The proof of our  results will be presented in the next section.

\section{Pseudo-multipliers in Schatten-von Neumann classes}

In this section we prove our main result for pseudo-multipliers $T_m$. Our criteria will be formulated in terms of the symbols $m.$ First, let us observe that every multiplier $T_m$ is an operator with kernel $K_m(x,y).$ In fact, straightforward computation show that
\begin{equation}\label{kernelpseudo}
T_mf(x) =\int_{\mathbb{R}^n}K_m(x,y)f(y)dy,\,\,K_m(x,y):=\sum_{\nu\in\mathbb{N}^n_0}m(x,\nu)\phi_\nu(x)\phi_{\nu}(y)
\end{equation}
for every  $f\in\mathscr{D}(\mathbb{R}^n).$ We will use the following result (see J. Delgado \cite{Delgado,D2}).
\begin{theorem}\label{Theorem1} Let us consider $1\leq p_1,p_2<\infty,$ $0<r\leq 1$ and let $p_1'$ be such that $\frac{1}{p_1}+\frac{1}{p_1'}=1.$ An operator $T:L^p(\mu_1)\rightarrow L^p(\mu_2)$ is $r$-nuclear if and only if there exist sequences $(g_n)_n$ in $L^{p_2}(\mu_2),$ and $(h_n)$ in $L^{q_1}(\mu_1),$ such that
\begin{equation}
\sum_{n}\Vert g_n\Vert_{L^{p_2}}^r\Vert h_n\Vert^r_{L^{q_1} } <\infty,\textnormal{        and        }Tf(x)=\int(\sum_ng_{n}(x)h_n(y))f(y)d\mu_1(y),\textnormal{   a.e.w. }x,
\end{equation}
for every $f\in {L^{p_1}}(\mu_1).$ In this case, if $p_1=p_2$ $($\textnormal{see Section 3 of} \cite{Delgado}$)$ the nuclear trace of $T$ is given by
\begin{equation}\label{trace1}
\textnormal{Tr}(T):=\int\sum_{n}g_{n}(x)h_{n}(x)d\mu_1(x).
\end{equation}
\end{theorem}
Now, we prove our main theorem.

\begin{proof}[Proof of Theorem \ref{MainTheorem}]
Let us consider a pseudo-multiplier $T_m.$ By definition, $T_m$ is a Hilbert-Schmidt operator if and only if there exists an orthonormal basis $\{e_\nu\}_\nu$ of $L^2(\mathbb{R}^n)$ such that
\begin{equation}
\sum_{\nu}\Vert T_m e_\nu \Vert_{L^2}^2<\infty.
\end{equation} In particular, if we choose the system of Hermite functions $\{\phi_\nu\},$ which provides an orthonormal basis of $L^2(\mathbb{R}^n),$ from the relation
$ T_m(\phi_\nu)=m(x,\nu)\phi_\nu ,$ we conclude that $T_m$ is of Hilbert-Schmidt type, if and only if
\begin{equation}
\sum_{\nu}\Vert m(\cdot,\nu)\phi_\nu \Vert_{L^2}^2=\sum_{\nu\in\mathbb{N}_0^n}\int_{\mathbb{R}^n}|m(x,\nu)|^2\phi_\nu(x)^2dx<\infty.
\end{equation} So, we have proved the first statement. Now, if we assume that $T_m$ is positive and self-adjoint, then $T_m$ is of class trace  if and only if there exists an orthonormal basis $\{e_\nu\}_\nu$ of $L^2(\mathbb{R}^n)$ such that
\begin{equation}
\sum_{\nu}\langle T_m e_\nu ,e_\nu\rangle_{L^2}<\infty.
\end{equation} As in the first assertion, if we choose the basis formed by the Hermite functions, $T_m$ is of class trace if and only if
\begin{equation}
\sum_{\nu}\langle T_m e_\nu ,e_\nu\rangle_{L^2}=\sum_{\nu\in\mathbb{N}_0^n}\int_{\mathbb{R}^n}m(x,\nu)\phi_\nu(x)^2dx<\infty,
\end{equation} which proves the second assertion. Now, we will verify that \eqref{0r1} implies that $T_m\in S_r(L^2)$ for $0<r\leq 1.$ For this, we will use Delgado's Theorem (Theorem \ref{Theorem1}) to the representation
\eqref{kernelpseudo} of $K_m$
\begin{equation}
K_m(x,y):=\sum_{\nu\in\mathbb{N}^n_0}m(x,\nu)\phi_\nu(x)\phi_{\nu}(y).
\end{equation}
So, $T_m\in S_{r}(L^2)$ if
\begin{equation}
\sum_\nu \Vert m(\cdot,\nu)\Vert_{L^2}^r\Vert \phi_\nu\Vert_{L^2}^r=\sum_{\nu\in\mathbb{N}_0^n}\left(\int_{\mathbb{R}^n}|m(x,\nu)|^2\phi_{\nu}(x)^2dx\right)^{\frac{r}{2}}<\infty,
\end{equation}
where we have used that the $L^2-$norm of every Hermite function $\phi_\nu$ is normalised. In order to finish the proof, we only need to prove that \eqref{1r2} assures that $T_m\in S_r(L^2)$ for $1<r<2.$ This can be proved by using the following multiplication property on Schatten classes:
\begin{equation}
S_p(H)S_q(H)\subset S_{r}(H),\,\,\,\frac{1}{r}=\frac{1}{p}+\frac{1}{q}.
\end{equation}
So, we will factorize $T_m$ as
\begin{equation}
T_{m}=T_{m}H^\sigma H^{-\sigma}, \,\,\,\sigma>0,
\end{equation} where $H$ is the harmonic oscillator. Let us note that the symbol of $A=T_m H^\sigma$ is given by $a(x,\nu)=m(x,\nu)(2|\nu|+n)^\sigma.$ So, from the second assertion, $A\in S_{2}(L^2)$ if and only if
$$ \sum_{\nu\in\mathbb{N}_0^n}|\nu|^{2\sigma}\int_{\mathbb{R}^n}|m(x,\nu)|^2\phi_\nu(x)^2dx\asymp   \sum_{\nu\in\mathbb{N}_0^n}(2|\nu|+n)^{2\sigma}\int_{\mathbb{R}^n}|m(x,\nu)|^2\phi_\nu(x)^2dx<\infty. $$ In order to prove that $T_m\in S_{r}{(L^2)},$ in view of the multiplication property
\begin{equation}
S_{2}(L^2)S_{\frac{2r}{2-r}}(L^2)\subset S_{r}(L^2),
\end{equation}
we only need to prove that $H^{-\sigma}\in S_{p}{(L^2)} $ with $p=\frac{2-r}{2r}.$  The symbol of $H^{-\sigma}$ is given by $a'(\nu)=(2|\nu|+n)^{-\sigma}.$ By using the hypothesis $\sigma>n(\frac{1}{r}-\frac{1}{2})$ we have that
$$\sum_{\nu} |a'(\nu)|^{p}=\sum_{\nu}(2|\nu|+n)^{-\sigma p}<\infty $$ because $\sigma p=\sigma(\frac{1}{r}-\frac{1}{2})^{-1}>n.$ So, we finish the proof.
\end{proof}

\subsection{Trace class pseudo-multipliers of the harmonic oscillator} In order to determinate a relation with the eigenvalues of $T_m$ we recall the following result (see \cite{O}).
\begin{theorem} Let $T:L^p(\mu)\rightarrow L^p(\mu)$ be a $r$-nuclear operator as in \eqref{nuc}. If $\frac{1}{r}=1+|\frac{1}{p}-\frac{1}{2}|,$ then, 
\begin{equation}
\textnormal{Tr}(T):=\sum_{n\in\mathbb{N}^n_0}e_n'(f_n)=\sum_{n}\lambda_n(T)
\end{equation}
where $\lambda_n(T),$ $n\in\mathbb{N}$ is the sequence of eigenvalues of $T$ with multiplicities taken into account. 
\end{theorem}
As an immediate consequence of the preceding theorem, if $T_m:L^2(\mathbb{R}^n)\rightarrow L^2(\mathbb{R}^n)$ is a trace class ($1$-nuclear)  then, 
\begin{equation}
\textnormal{Tr}(T_m)=\int_{\mathbb{R}^n}\sum_{\nu\in\mathbb{N}_0^n}m(x,\nu)\phi_\nu(x)^2dx.=\sum_{n}\lambda_n(T),
\end{equation}
where $\lambda_n(T),$ $n\in\mathbb{N}$ is the sequence of eigenvalues of $T_m$ with multiplicities taken into account.\\
\\


\bibliographystyle{amsplain}

\end{document}